\pdfoutput=1
\documentclass{louloupart1}
\usepackage{mathtools}
\usepackage{graphicx}
\usepackage{hyperref}
\usepackage{bm}
\usepackage{tikz}
\usepackage{titlesec}
\usepackage{centernot}

\title[Acylindrical hyperbolicity and property $R_\infty$ for Higman groups]
{Acylindrical hyperbolicity of automorphism groups and property $R_\infty$ for generalized Higman groups}
\author[I Soroko]{Ignat Soroko}
\givenname{Ignat}
\surname{Soroko}
\address{Ignat Soroko, Department of Mathematics, Southern Methodist University, Dallas, TX 75205, USA}
\email{isoroko@smu.edu}
\author[N Vaskou]{Nicolas Vaskou}
\givenname{Nicolas}
\surname{Vaskou}
\address{Nicolas Vaskou, Section de math\'ematiques, Universit\'e de Gen\`eve, Rue du Conseil-G\'en\'eral 7--9 1205, Gen\`eve, Switzerland}
\email{nicolas.vaskou@unige.ch}
\date{\today}

\newtheorem{thm}{Theorem} 
\newtheorem{lem}[thm]{Lemma}
\newtheorem*{delzantlem}{Delzant's Lemma}
\newtheorem*{martincorg}{Martin's Corollary G}
\newtheorem*{ff-criterion}{Corollary~\ref*{cor:aut-ah-implies-rinfty}}
\newtheorem{prop}[thm]{Proposition}
\newtheorem{corl}[thm]{Corollary}

\theoremstyle{definition}

\newtheorem{rem}[thm]{Remark}

\newtheorem*{acknow}{Acknowledgments}

\titleformat{\subsection}
  {\normalfont\normalsize\bfseries}
  {\thesubsection}{1em}{}

\begin{document}

\def\N{\mathbb N} \def\Inn{{\rm Inn}} \def\Out{{\rm Out}} 
\def\id{{\rm id}} \def\supp{{\rm supp}} 
\renewcommand{\Im}{\operatorname{Im}} 
\def\Ker{{\rm Ker}} \def\PP{\mathcal P} \def\Homeo{{\rm Homeo}}
\def\SHomeo{{\rm SHomeo}} \def\LHomeo{{\rm LHomeo}}
\def\MM{\mathcal M} \def\CC{\mathcal C} \def\AA{\mathcal A}
\def\S{\mathbb S} 
\def\FF{\mathcal F} \def\SS{\mathcal S}
\def\LL{\mathcal L} \def\D{\mathbb D} 
\renewcommand{\H}{\mathbb H}

\newcommand{\Hig}{\operatorname{Hig}}
\newcommand{\Cyl}{\operatorname{Cyl}}
\newcommand{\Aut}{\operatorname{Aut}}
\newcommand{\SL}{\operatorname{SL}}
\newcommand{\GL}{\operatorname{GL}}
\newcommand{\SO}{\operatorname{SO}}
\newcommand{\PSO}{\operatorname{PSO}}
\newcommand{\PSL}{\operatorname{PSL}}
\newcommand{\PGL}{\operatorname{PGL}}
\newcommand{\conj}{\operatorname{conj}}
\newcommand{\card}{\operatorname{Card}}
\newcommand{\Sym}{\operatorname{Sym}}
\newcommand{\Cent}{\operatorname{Cent}}
\newcommand{\Fix}{\operatorname{Fix}}
\newcommand{\Isom}{\operatorname{Isom}}
\newcommand{\Cay}{\operatorname{Cay}}
\newcommand{\Stab}{\operatorname{Stab}}
\newcommand{\CAT}{\operatorname{CAT}}
\newcommand{\Hom}{\operatorname{Hom}}
\newcommand{\ov}{\overline}
\mathchardef\mhyphen="2D
\renewcommand{\le}{\leqslant}
\renewcommand{\ge}{\geqslant}
\newcommand{\lk}{\operatorname{lk}}
\newcommand{\im}{\operatorname{im}}

\begin{abstract}
We consider the Higman groups $H_k$ ($k\ge 4$) and their generalizations studied by Martin and Horbez--Huang. We prove that the automorphism groups of the Horbez--Huang groups on at least five generators, as well as those of Martin's $4$-generated Higman-like groups, are acylindrically hyperbolic. It follows that the groups themselves are acylindrically hyperbolic, admit infinite-dimensional spaces of automorphism-invariant homogeneous quasimorphisms, have infinite $2$-generated simple characteristic quotients, and satisfy property $R_\infty$. We also prove, using Delzant's lemma, a general criterion ensuring that an individual automorphism has infinite Reidemeister number: if $G$ is a group, $\varphi\in\Aut(G)$, the group $\Gamma_\varphi=\langle\Inn(G),\varphi\rangle$ acts by isometries on a Gromov-hyperbolic space, and the restricted action of $\Inn(G)$ has two independent loxodromic elements, then $R(\varphi)=\infty$.

\smallskip\noindent
{\bf 2020 Mathematics Subject Classification~} Primary: 20F65, 20F28; Secondary: 20F67, 20E45.

\smallskip\noindent
{\bf Keywords\ \ } Acylindrical hyperbolicity, automorphism groups, generalized Higman groups, property $R_\infty$, twisted conjugacy.
\end{abstract}

\maketitle


\section{Introduction}\label{sec1}

The Higman groups form a remarkable family at the crossroads of combinatorial and geometric group theory. For $k\ge4$, the classical Higman group is defined as
\[
H_k=\left\langle a_1,\dots,a_k\ \middle|\ a_i a_{i+1} a_i^{-1}=a_{i+1}^2 \text{ for all } i\in \Z/k\Z \right\rangle .
\]
Introduced by Higman in~\cite{Higma1}, they were the first infinite finitely presented groups without nontrivial finite quotients. For a long time they were viewed mainly as a source of pathological examples, but recent research revealed a rich geometric picture.

In~\cite{Marti1}, Martin introduced the $4$-generated family of \emph{Higman-like groups}
\[
H(m_i)_{i\in \Z/4\Z}=\left\langle a_1,a_2,a_3,a_4 \ \middle|\ a_i a_{i+1} a_i^{-1}=a_{i+1}^{m_i} \ \text{for all } i\in\Z/4\Z \right\rangle,\qquad m_i\ge2,
\]
which includes the classical Higman group $H_4$ as the special case $m_1=m_2=m_3=m_4=2$. He showed that these groups admit natural cocompact actions on associated $\CAT(0)$ square complexes, and used this geometry to study their endomorphisms and automorphisms. For the higher Higman groups, the picture becomes even richer: for $k\ge5$, the same cyclic pattern in the defining presentation gives rise to negatively curved polygonal complexes, and Martin exhibited a natural $\CAT(-1)$ space on which $H_k$ acts acylindrically~\cite{Marti2}.

More recently, Horbez and Huang~\cite{HorHua1} considered \emph{generalized Higman groups}. Let $k\ge4$ and let $\sigma=\bigl((m_1,n_1),\dots,(m_k,n_k)\bigr)$, where $m_i,n_i\in\Z\setminus\{0\}$ satisfy $|m_i|\neq |n_i|$ for all $i$. They define
\[
H_\sigma=\left\langle a_1,\dots,a_k \ \middle|\ a_i a_{i+1}^{m_i} a_i^{-1}=a_{i+1}^{n_i} \ \text{for } i\in \Z/k\Z \right\rangle.
\]
This family contains the classical Higman groups as the special case
\[
\sigma_k=\bigl((1,2),\dots,(1,2)\bigr),\qquad H_{\sigma_k}=H_k,
\]
and it also contains Martin's $4$-generated Higman-like groups as the subfamily
\[
\sigma=\bigl((1,m_1),(1,m_2),(1,m_3),(1,m_4)\bigr),\qquad m_i\ge2.
\]
Geometrically, the groups $H_\sigma$ arise from the same cyclic polygon-of-groups construction as the classical Higman groups, except that the Baumslag--Solitar parameters are allowed to vary from one vertex of the cycle to the next. Horbez and Huang~\cite{HorHua1} showed that the associated complexes determine the groups $H_\sigma$ themselves. For $k\ge5$, this yields measure equivalence rigidity, as well as consequences for lattice embeddings and automorphism groups.

The main result of this paper concerns the acylindrical hyperbolicity of the automorphism groups of these groups. Recall that a group is \emph{acylindrically hyperbolic} if it admits a non-elementary acylindrical action on a geodesic Gromov-hyperbolic space. The precise definitions and terminology used here are recalled in Sections~\ref{sec2} and~\ref{sec3}.

\begin{thm}\label{thm1}
Let $G$ belong to one of the following two families.
\begin{enumerate}
\item $G=H_\sigma$ is a generalized Higman group on $k\ge5$ generators,
\[
H_\sigma=\left\langle a_1,\dots,a_k \ \middle|\ a_i a_{i+1}^{m_i}a_i^{-1}=a_{i+1}^{n_i} \ \text{for }i\in\Z/k\Z \right\rangle,
\]
where $m_i,n_i\in\Z\setminus\{0\}$ and $|m_i|\ne |n_i|$ for every $i$.
\item $G=H(m_i)_{i\in\Z/4\Z}$ is a $4$-generated Higman-like group in the sense of Martin,
\[
H(m_i)_{i\in\Z/4\Z}=\left\langle a_1,a_2,a_3,a_4 \ \middle|\ a_i a_{i+1}a_i^{-1}=a_{i+1}^{m_i} \ \text{for }i\in\Z/4\Z \right\rangle,\qquad m_i\ge2.
\]
\end{enumerate}
Then $\Aut(G)$ is acylindrically hyperbolic.
\end{thm}

Acylindrical hyperbolicity of $\Aut(G)$ has several striking consequences. We record below three of the most remarkable ones. Recall that a function $q\colon G\to\R$ is a \emph{quasimorphism} if its defect
$
D(q):=\sup_{g,h\in G}|q(gh)-q(g)-q(h)|
$
is finite. It is \emph{homogeneous} if $q(g^n)=nq(g)$ for every $g\in G$ and $n\in\Z$, and it is $\Aut(G)$-invariant if $q\circ\alpha=q$ for every $\alpha\in\Aut(G)$.

\begin{corl}\label{cor:aut-consequences}
Let $G$ be any group appearing in Theorem~\ref{thm1}. Then:
\begin{enumerate}
\item $G$ is acylindrically hyperbolic;
\item the vector space of $\Aut(G)$-invariant homogeneous quasimorphisms on $G$ is infinite-dimensional;
\item $G$ admits an infinite, $2$-generated, simple characteristic quotient.
\end{enumerate}
\end{corl}

The first conclusion recovers and extends earlier results about the groups themselves. Acylindrical hyperbolicity of the classical Higman groups was established by Minasyan--Osin~\cite{MinOsi1}. For $k\ge5$, Martin proved that the natural action of $H_k$ on its associated $\CAT(-1)$ complex is acylindrical~\cite[Corollary~5.2]{Marti2}; we show the non-elementarity of this action. 

The second conclusion places these groups among those possessing a large supply of homogeneous quasimorphisms invariant under every automorphism. Besides their application to property $R_\infty$, discussed below, such quasimorphisms detect unbounded $\Aut(G)$-invariant norms and have other applications discussed in~\cite[Section~3]{FFaWad1}.

The third conclusion is particularly well suited to the historical origin of Higman's group. Higman introduced the four-generator group $H_4$ as an auxiliary group in the construction of the first finitely generated infinite simple group: $H_4$ is infinite and has no nontrivial finite quotients, so its quotient by any maximal proper normal subgroup is an infinite finitely generated simple group~\cite{Higma1}. Corollary~\ref{cor:aut-consequences}(3), which follows from the general result~\cite[Corollary~3.7]{CouFou1}, shows that the kernel can moreover be chosen characteristic and the quotient $2$-generated. The same conclusion holds for every classical Higman group $H_k$, $k\ge4$, and more generally for every group covered by Theorem~\ref{thm1}.

Theorem~\ref{thm1} contributes to a question of Genevois~\cite[Question~1.1]{GenOne1}, asking whether the automorphism group of every finitely generated acylindrically hyperbolic group is itself acylindrically hyperbolic; see also~\cite[Problem~21.49]{Kourovka1}. A positive answer is known for one-ended hyperbolic groups~\cite{GenOne1}, finitely generated groups with infinitely many ends~\cite{GenHor1}, and infinite irreducible non-affine strongly rigid Coxeter groups~\cite{Sorok1}. Under additional hypotheses, positive results have also been obtained for broad classes of relatively hyperbolic groups~\cite{GenHor1}, certain classes of graph products~\cite{GenMar1,Genev1,EscHor1}, and some Artin groups~\cite{JonManSar1}.

There is also a connection with the finite index overgroup problem. Martin~\cite[Proposition~3.10]{Marti1} and Horbez--Huang~\cite[Corollary~5.21(1)]{HorHua1} show that, for the groups in Theorem~\ref{thm1}, the outer automorphism group is finite. Thus $G\simeq\Inn(G)$ has finite index in $\Aut(G)$. In view of Corollary~\ref{cor:aut-consequences}(1), Theorem~\ref{thm1} gives positive examples for the open question asking whether acylindrical hyperbolicity passes to finite index overgroups. This assertion was originally stated in~\cite[Lemma~3.9]{MinOsi1}, and was later corrected to an open question in~\cite[Question~2]{MinOsiCorr1}; see also the work of Balasubramanya~\cite[Question~1.2]{Balas1}, where she gives an affirmative answer when the finite index subgroup is normal and $\mathcal{AH}$-accessible.

We now turn to twisted conjugacy and property $R_\infty$. Let $G$ be a group and let $\varphi$ be an automorphism of $G$. We say that elements $g,h\in G$ are \emph{$\varphi$-twisted conjugate} if there exists $x\in G$ such that $h=x\,g\,\varphi(x)^{-1}$. This defines an equivalence relation on $G$, and the (possibly infinite) number of its equivalence classes is called the \emph{Reidemeister number} of $\varphi$, denoted by $R(\varphi)$. We say that $G$ \emph{has property $R_\infty$} if $R(\varphi)=\infty$ for all $\varphi\in \Aut(G)$.

Property $R_\infty$ is a subtle and somewhat mysterious property of groups, which has its origins in Nielsen fixed point theory, and the problem of determining which classes of groups possess it remains an area of active research. The list of groups known to have property $R_\infty$ is quite large and contains non-elementary Gromov-hyperbolic groups~\cite{Felsh1} and wide classes of relatively hyperbolic groups~\cite{FFIMSW1}, non-abelian generalized Baumslag--Solitar groups, many weakly branch groups, many arithmetic linear groups, mapping class groups of surfaces with large enough complexity, and some other non-amenable groups; see~\cite{FelTro2,FelNas1} for references. On the other hand, the free nilpotent group $N_{r,c}$ of rank $r$ and nilpotency class $c$ has property $R_\infty$ if and only if $c\ge2r$; see~\cite{DekGon1}. It has been conjectured that all finitely generated residually finite groups which are not solvable-by-finite, or, more weakly, all finitely generated non-amenable residually finite groups, have property $R_\infty$~\cite{FelTro2,SorVas1}. For some recent developments, see~\cite{CalSor1,Witdo1,DekLat1,Troit1,FFIMSW1,DeLiRe1,Sorok1}.

Our proof also establishes the following general criterion. We use the Bestvina--Fujiwara convention for the term ``non-elementary''; the exact terminology, and the reason for this choice, are discussed in Section~\ref{sec2}.

\begin{prop}\label{prop:delzant}
Let $G$ be a group and let $\varphi\in\Aut(G)$. Set
\[
\Gamma_\varphi=\langle\Inn(G),\varphi\rangle\le\Aut(G).
\]
Assume that $\Gamma_\varphi$ acts by isometries on a geodesic Gromov-hyperbolic space $X$ and that the restricted action of $\Inn(G)$ is non-elementary. Then $R(\varphi)=\infty$.
\end{prop}

Proposition~\ref{prop:delzant} strengthens the criterion of Fournier-Facio and Wade~\cite[Theorem~E]{FFaWad1} and its version for individual automorphisms in~\cite[Remark following Corollary~8.1.6]{FFIMSW1}. Indeed, instead of requiring $\Aut(G)$ or $\Gamma_\varphi$ to be acylindrically hyperbolic, it suffices to have a hyperbolic action of $\Gamma_\varphi$ whose restriction to $\Inn(G)$ is non-elementary.

\begin{corl}\label{cor:main-rinfty}
Every group $G$ appearing in Theorem~\ref{thm1} has property $R_\infty$. In particular, the classical Higman groups $H_k$ have property $R_\infty$ for all $k\ge4$.
\end{corl}

Horbez--Huang's developed complex and combinatorial rigidity theorem also apply to $4$-generated generalized Higman groups~\cite[Theorem~4.2]{HorHua1}. However, their description of the automorphism group is proved only for $k\ge5$~\cite[Corollary~5.21]{HorHua1}, while for $k=4$ such a description is known only for Martin's subfamily. Thus our method does not cover the remaining $4$-generated generalized Higman groups. Monod~\cite{Monod1} introduced other variations of Higman's construction. Our approach does not directly extend to these groups, since the necessary results on their geometry and automorphisms are not presently known.

Our paper is structured as follows. In Section~\ref{sec2}, we prove Proposition~\ref{prop:delzant} using Delzant's lemma, derive a general acylindrical-hyperbolicity criterion for property $R_\infty$, and explain our convention for non-elementary actions. In Section~\ref{sec3}, we introduce the Martin--Horbez--Huang complex $X_\sigma$, describe the natural action of $\Aut(G)$ on it, and establish the geometric properties needed later. In Section~\ref{sec4}, we prove Theorem~\ref{thm1}: for $k\ge5$, it follows from the hyperbolicity of $X_\sigma$ and the acylindricity of the induced action, while for Martin's $4$-generated family we apply his criterion for non-uniformly weakly acylindrical actions on $\CAT(0)$ square complexes. 
Finally, in Section~\ref{sec:remarks} we compare several general conditions implying property $R_\infty$ and give examples showing that some of these conditions do not imply others.

\begin{acknow}
The authors thank Alexandre Martin for useful conversations, and Francesco Fournier-Facio and the anonymous referee for a careful reading and valuable suggestions which improved the quality of our paper. This work was started at the AIM workshop ``Geometry and topology of Artin groups'' organized by Ruth Charney, Kasia Jankiewicz and Kevin Schreve in September 2023, and continued at the BIRS workshop ``Combinatorial Nonpositive Curvature'' organized by Kasia Jankiewicz and Piotr Przytycki in September 2024. The authors thank the organizers, AIM and BIRS. The second author is supported by the fellowship \#P5R5PT\textunderscore230599/1 of the Swiss National Science Foundation.
\end{acknow}

\section{Twisted conjugacy and Delzant's lemma}\label{sec2}

Recall that for a group $G$ with center $Z(G)$, the subgroup $\Inn(G)\le \Aut(G)$ of inner automorphisms is isomorphic to $G/Z(G)$ via $g\mapsto \conj_g$, where $\conj_g(x)=gxg^{-1}$.

Let $\Gamma$ be a group acting by isometries on a geodesic Gromov-hyperbolic space $X$, and let $\partial X$ denote its Gromov boundary. Following Bestvina--Fujiwara~\cite{BesFuj1}, we call the action \emph{non-elementary} if $\Gamma$ contains two independent loxodromic elements, that is, two loxodromic elements with disjoint fixed-point sets in $\partial X$. It is also called an action \emph{of general type} in some sources. 

Our main tool is the following Delzant's lemma. It first appeared in~\cite[Lemma~3.4]{LevLus1}, for hyperbolic groups and with a brief proof sketch, and was later stated in the present generality in~\cite[Lemma~6.3]{FeGoDa1}. For a detailed proof, see~\cite[Section~5]{SorVas1}.

\begin{delzantlem}
Let $\Gamma$ be a group acting non-elementarily by isometries on a geodesic Gromov-hyperbolic space, and let $K$ be a normal subgroup of $\Gamma$ such that the quotient $\Gamma/K$ is abelian. Then every coset of $K$ contains infinitely many conjugacy classes.
\end{delzantlem}

\begin{trivlist}
\item[\hskip\labelsep\textbf{Proposition~\ref*{prop:delzant}.}]\itshape
Let $G$ be a group and let $\varphi\in\Aut(G)$. Set
\[
\Gamma_\varphi=\langle\Inn(G),\varphi\rangle\le\Aut(G).
\]
Assume that $\Gamma_\varphi$ acts by isometries on a geodesic Gromov-hyperbolic space $X$ and that the restricted action of $\Inn(G)$ is non-elementary. Then $R(\varphi)=\infty$.
\end{trivlist}

\begin{proof}
Since the action of $\Inn(G)$ on $X$ is non-elementary, so is the action of $\Gamma_\varphi$ on $X$. Moreover, $\Inn(G)$ is a normal subgroup of $\Gamma_\varphi$, and the quotient $\Gamma_\varphi/\Inn(G)$ is cyclic, hence abelian. Therefore Delzant's lemma applies and shows that the coset $\Inn(G)\,\varphi$ contains infinitely many conjugacy classes in $\Gamma_\varphi$.

Consider the map (for clarity, we write multiplication in $\Aut(G)$ as $\circ$):
\[
[g]_\varphi \longmapsto [\conj_g\circ\,\varphi]
\]
from $\varphi$-twisted conjugacy classes in $G$ to ordinary conjugacy classes in $\Gamma_\varphi$ contained in the coset $\Inn(G)\,\varphi$. We first check that this map is well defined. Suppose that $g,h\in G$ are $\varphi$-twisted conjugate, so that $h=x\,g\,\varphi(x)^{-1}$ for some $x\in G$. Then for every $y\in G$,
\begin{multline*}
\bigl(\conj_x\circ\,(\conj_g\circ\,\varphi)\circ\conj_x^{-1}\bigr)(y)
=x\,g\,\varphi(x^{-1}yx)\,g^{-1}x^{-1}\\
=x\,g\,\varphi(x)^{-1}\,\varphi(y)\,\varphi(x)\,g^{-1}x^{-1}
=h\,\varphi(y)\,h^{-1}
=(\conj_h\circ\,\varphi)(y).
\end{multline*}
Hence
\[
\conj_x\circ\,(\conj_g\circ\,\varphi)\circ\conj_x^{-1}
=\conj_h\circ\,\varphi,
\]
so $\conj_g\circ\,\varphi$ and $\conj_h\circ\,\varphi$ are conjugate in $\Gamma_\varphi$. Thus the map is well defined.

It is also surjective, because every element of $\Inn(G)\,\varphi$ has the form $\conj_g\circ\,\varphi$ for some $g\in G$. Since the coset $\Inn(G)\,\varphi$ contains infinitely many conjugacy classes, it follows that $R(\varphi)=\infty$.

This proves $R(\varphi)=\infty$.
\end{proof}

\begin{corl}\label{cor:action-rinfty}
Let $G$ be a group. Assume that $\Aut(G)$ acts by isometries on a geodesic Gromov-hyperbolic space and that the restricted action of $\Inn(G)$ is non-elementary. Then $G$ has property $R_\infty$.
\end{corl}

\begin{proof}
Apply Proposition~\ref{prop:delzant} to every $\varphi\in\Aut(G)$, after restricting the given action to $\Gamma_\varphi$.
\end{proof}

As an application, we now give an independent proof of the criterion due to Fournier-Facio and collaborators (\cite[Theorem~E]{FFaWad1} and \cite[Corollary~8.1.4]{FFIMSW1}):

\begin{corl}\label{cor:aut-ah-implies-rinfty}
Let $G$ be a group such that $\Inn(G)$ is infinite and $\Aut(G)$ is acylindrically hyperbolic. Then $G$ has property $R_\infty$.
\end{corl}

\begin{proof}
Since $\Aut(G)$ is acylindrically hyperbolic, it admits a non-elementary acylindrical action on a geodesic Gromov-hyperbolic space $X$. The subgroup $\Inn(G)$ is infinite and normal in $\Aut(G)$, hence $s$-normal in the sense of~\cite{Osin1}. Therefore, by~\cite[Lemma~7.2]{Osin1}, the restricted action of $\Inn(G)$ on $X$ is non-elementary. Corollary~\ref{cor:action-rinfty} now implies that $G$ has property $R_\infty$.
\end{proof}

\begin{rem}\label{rem:non-elementary}
There are two notions of a \emph{non-elementary action} in the literature. In the sense adopted by Bestvina--Fujiwara, an action on a hyperbolic space is non-elementary if it contains two independent loxodromic elements~\cite[p.~72]{BesFuj1}. By contrast, following Gromov~\cite[Section~8.2]{Gromo1}, one often calls an action elementary if its limit set in the boundary has cardinality at most $2$, and hence non-elementary if the limit set has cardinality greater than $2$; this is the terminology adopted, for instance, by Osin~\cite[p.~852]{Osin1}. The discrepancy comes from the \emph{quasi-parabolic case}: by definition, a quasi-parabolic action fixes a unique point of the boundary and contains a loxodromic element. In particular, a quasi-parabolic action is non-elementary in Gromov's sense, but all of its loxodromic elements share the same fixed point at infinity, so there are no independent loxodromics. For acylindrical actions the two notions of non-elementary action coincide, since Osin's trichotomy rules out quasi-parabolic actions~\cite[Theorem~1.1]{Osin1}.

This distinction is relevant for Delzant's lemma. Consider the Baumslag--Solitar group $\Gamma=BS(1,2)=\langle a,b\mid bab^{-1}=a^2\rangle$, realized as a subgroup of $\GL_2(\R)$ by matrices $a=\left(\begin{smallmatrix}1&1\\0&1\end{smallmatrix}\right)$ and $b=\left(\begin{smallmatrix}2&0\\0&1\end{smallmatrix}\right)$, acting on $\mathbb H^2$ by fractional linear transformations. Every element of $\Gamma$ fixes $\infty$, since all matrices are upper triangular, while $b$ is loxodromic, as it fixes $0$ and $\infty$; thus the action is quasi-parabolic. In particular, it is non-elementary in Gromov's sense, since its limit set is infinite, but not in the stronger Bestvina--Fujiwara sense, since all loxodromic elements share the fixed point $\infty$ and therefore there are no independent loxodromics. Moreover, if $K=\ker(\Gamma\to\Z)$, where $a\mapsto 0$ and $b\mapsto 1$, then $K\lhd \Gamma$, and $\Gamma/K\cong\Z$. An easy computation shows that $b^{-m}a^nb^{m}=\left(\begin{smallmatrix}1& {n}/{2^m}\\0&1\end{smallmatrix}\right)$, hence $K$ consists of the unipotent matrices $\left(\begin{smallmatrix}1&r\\0&1\end{smallmatrix}\right)$ with $r\in \Z[1/2]$, and the coset $Kb$ consists of all matrices of the form $\left(\begin{smallmatrix}2&r\\0&1\end{smallmatrix}\right)$ with $r\in \Z[1/2]$. However, this coset contains only one conjugacy class: because conjugation by $\left(\begin{smallmatrix}1&s\\0&1\end{smallmatrix}\right)\in K$ sends $\left(\begin{smallmatrix}2&r\\0&1\end{smallmatrix}\right)$ to $\left(\begin{smallmatrix}2&r-s\\0&1\end{smallmatrix}\right)$, all elements of $Kb$ are conjugate. Hence Delzant's lemma conclusion fails in this quasi-parabolic setting, which is why we adopt the Bestvina--Fujiwara definition of a non-elementary action for our needs.
\end{rem}

\section{The Martin--Horbez--Huang complex \texorpdfstring{$X_\sigma$}{X\_sigma} and the action of \texorpdfstring{$\Aut(G)$}{Aut(G)}}\label{sec3}

In this section we collect the geometric and algebraic facts that will be used later for both generalized Higman groups with $k\ge 5$ and the $4$-generated Higman-like groups of Martin.

Throughout, let
\[
G=H_\sigma,
\]
where either:
\begin{itemize}
\item $k\ge 5$ and
\[
\sigma=\bigl((m_1,n_1),\dots,(m_k,n_k)\bigr),
\qquad |m_i|\neq |n_i| \ \text{\ \ for all } i,
\]
or
\item $k=4$ and
\[
\sigma=\bigl((1,m_1),(1,m_2),(1,m_3),(1,m_4)\bigr),
\qquad m_i>1 \ \text{\ \ for all } i.
\]
\end{itemize}
Thus the $4$-generated Higman-like groups considered by Martin are precisely the groups $H_\sigma$ in the second case, and the classical Higman group is
\[
H_4=H_{((1,2),(1,2),(1,2),(1,2))}.
\]

\subsection{The complex \texorpdfstring{$X_\sigma$}{X\_sigma}}

Consider the complex of groups over the $k$-gon $K$ constructed as follows. Let $K$ be a $2$-dimensional $k$-gon with vertices $x_i$ and edges $e_i$, where $e_i$ is oriented from $x_i$ to $x_{i+1}$ for $i\in \Z/k\Z$. We equip $K$ with a complex-of-groups structure as follows. To each vertex $x_i$ we assign the Baumslag--Solitar group
\[
G_{x_i}=
\langle a_{i-1},a_i \mid
a_{i-1}a_i^{m_{i-1}}a_{i-1}^{-1}=a_i^{n_{i-1}}\rangle
\simeq BS(m_{i-1},n_{i-1}),
\]
and to each edge $e_i$ we assign the cyclic group
\[
G_{e_i}=\langle a_i\rangle.
\]
To the $2$-cell $K$ itself we assign the trivial group. Thus this is a simple complex of groups in the sense of~\cite[Chapter~II.12]{BriHae1}. By Britton's Lemma, the natural homomorphisms from $G_{e_i}$ to the adjacent vertex groups, sending the generator of $G_{e_i}$ to $a_i$, are injective. In the generalized Higman setting, Horbez and Huang proved that the natural homomorphisms $G_{x_i}\to G$ are injective \cite[Lemma~2.1]{HorHua1}; in the $4$-generated Higman-like setting, this is the developability of Martin's square of groups \cite[Section~1.1]{Marti1}. Thus in both cases we may regard the vertex groups $G_{x_i}$, and hence also the edge groups $G_{e_i}$, as subgroups of $G$.

Following Horbez--Huang, and Martin in the $k=4$ case, let $X=X_\sigma$ be the simply connected polygonal complex obtained by development of the complex of groups defined by $K$ and the above vertex groups, edge groups and homomorphisms (see \cite[Theorem II.12.18]{BriHae1}). Its vertices are the left cosets $gG_{x_i}$ with $g\in G$, its edges are the left cosets $gG_{e_i}$ with $g\in G$, and its $2$-cells correspond to the left cosets of the trivial subgroup of $G$, that is, they are indexed by the elements $g\in G$. Incidence is defined by inclusion of cosets. By \cite[Lemma~2.2]{HorHua1}, the quotient $G\backslash X$ is naturally identified with $K$; for $k=4$, this is exactly Martin's square complex associated to the standard square of groups \cite[Section~1.1]{Marti1}.

Since the quotient $G\backslash X$ is the finite complex $K$, the complex $X$ has finitely many isometry types of cells. By \cite[Section~2.2]{HorHua1} and~\cite[Theorem~I.7.50]{BriHae1}, the complex $X$ admits a $G$-invariant metric for which:
\begin{itemize}
\item if $k=4$, then $X$ is a complete geodesic $\CAT(0)$ piecewise Euclidean polygonal complex;
\item if $k\ge5$, then $X$ is a complete geodesic $\CAT(-1)$ piecewise hyperbolic polygonal complex (with each face isometric to a polygon in $\H^2$).
\end{itemize}
For $k=4$, this is precisely Martin's $\CAT(0)$ square complex. In all cases, the action of $G$ on $X$ is by cellular isometries without inversion, that is, every element stabilizing a cell fixes it pointwise \cite[Section~2.2]{HorHua1}, \cite[Section~1.1]{Marti1}.

\subsection{The group \texorpdfstring{$\Aut(G)$}{Aut(G)} and its action}\label{subsec:aut-action}

We now explain how $\Aut(G)$ acts on $X$ and how this relates to the original action of $G$ on $X$.

In the generalized Higman case, there is no loss of generality in assuming that $|m_i|<|n_i|$ for every $i$. Indeed, as observed by Horbez--Huang, replacing a pair $(m_i,n_i)$ by $(n_i,m_i)$ in $\sigma$ yields an isomorphic group, via the isomorphism sending the corresponding generator $a_i$ to $a_i^{-1}$.

Recall that
\[
\sigma=\bigl((m_1,n_1),\dots,(m_k,n_k)\bigr),
\]
with indices taken modulo $k$. Following Horbez--Huang, let $F_\sigma$ be the finite subgroup of translations
\[
\tau\colon \Z/k\Z\to \Z/k\Z,\qquad i\mapsto i+r,
\]
for some $r\in \Z/k\Z$, such that for every $i$, one has
\[
(m_{\tau(i)},n_{\tau(i)})=(m_i,n_i)
\qquad\text{or}\qquad
(m_{\tau(i)},n_{\tau(i)})=(-m_i,-n_i).
\]
In the $4$-generated Higman-like case, this is exactly the finite group $K$ of cyclic index symmetries preserving the tuple $(m_i)$ that appears in Martin's description of $\Aut(G)$.

Each $\tau\in F_\sigma$ induces an automorphism of $G$ by permuting the standard generators according to $i\mapsto \tau(i)$.

\medskip

\noindent
\textbf{The generalized Higman case.}
Set
\[
\widehat G:=G\rtimes F_\sigma.
\]
Horbez and Huang prove that the action of $G$ on $X$ extends naturally to an action of $\widehat G$ on $X$, and that the resulting homomorphism
\[
\widehat G\longrightarrow \Aut(X)
\]
is an isomorphism \cite[Theorem~4.2]{HorHua1}. Under the normalization $|m_i|<|n_i|$, they also prove that the natural homomorphism
\[
\widehat G\longrightarrow \Aut(G),
\qquad
(g,\tau)\longmapsto \conj_g\circ\, \tau,
\]
is an isomorphism \cite[Corollary~5.21]{HorHua1}.

\medskip

\noindent
\textbf{The $4$-generated Higman-like case.}
Martin proves that every nontrivial endomorphism of $G$ induces a combinatorial self-map of $X$ restricting to an orientation-preserving isomorphism on each square \cite[Proposition~3.4]{Marti1}, and that
\[
\Aut(G)\simeq G\rtimes F_\sigma,
\]
see \cite[Proposition~3.10]{Marti1}.

\medskip

Thus, in both cases, $\Aut(G)$ acts naturally on $X$ by cellular isometries.

Finally, in the generalized Higman case Horbez and Huang prove that every nontrivial conjugacy class in $G$ is infinite \cite[Lemma~2.6]{HorHua1}, and hence no nontrivial element can be central, so $Z(G)=1$; in the $4$-generated Higman-like case, Martin proves that $Z(G)=1$ \cite[Lemma~3.1 and Corollary~3.2]{Marti1}. Hence in both cases we have $Z(G)=1$ and $ G\simeq \Inn(G)$.

Under this identification, the action of $\Inn(G)$ on $X$ is exactly the original action of $G$ on $X$; for the generalized Higman case this follows from the definition of the homomorphism $\widehat G\to \Aut(G)$, and for the $4$-generated Higman-like case it is \cite[Lemma~3.3]{Marti1}.

\subsection{Common lemmas on the action}

We now record several facts that will be used later.

Recall that an action of a group $G$ by isometries on a metric space $X$ is called \emph{non-uniformly weakly acylindrical} if there exists a constant $R>0$ such that for any two points $x,y\in X$ with $d_X(x,y)\ge R$, the common stabilizer of $x$ and $y$ is finite. (We retain Martin's terminology~\cite{Marti2}; here ``non-uniformly'' means that no uniform bound on the cardinalities of these finite common stabilizers is assumed.) An action is called \emph{weakly acylindrical} if there exist constants $R,N>0$ such that for any two points $x,y\in X$ with $d_X(x,y)\ge R$, the common stabilizer of $x$ and $y$ has cardinality at most $N$. Finally, an action is called \emph{acylindrical} if for every $\varepsilon>0$ there exist $R,N>0$ such that whenever $x,y\in X$ satisfy $d_X(x,y)\ge R$, there are at most $N$ elements $g\in G$ with $d_X(x,gx)\le \varepsilon$ and $d_X(y,gy)\le \varepsilon$. Clearly, an acylindrical action is weakly acylindrical, and a weakly acylindrical action is non-uniformly weakly acylindrical.

\begin{lem}\label{lem:Aut-weak-acyl}
The action of $\Aut(G)$ on $X$ is weakly acylindrical.
\end{lem}

\begin{proof}
In the generalized Higman case, the action of $G$ on $X$ is weakly acylindrical by \cite[Lemma~2.7]{HorHua1}. In the $4$-generated Higman-like case, let $h\in G$ be nontrivial and suppose that $h$ fixes two points $x,y\in X$. Then \cite[Corollary~2.3]{Marti1} implies that $\Fix(h)$ is either a single vertex or a union of edges contained in the star of a single vertex. In particular, $\Fix(h)$ has diameter at most $2$. Therefore, if $d(x,y)>2$, no nontrivial element of $G$ can fix both $x$ and $y$. Thus the action of $G$ on $X$ is weakly acylindrical in this case as well. Let $R,N> 0$ be the corresponding constants of weak acylindricity for $G$.

Let $x,y\in X$ be any two points with $d(x,y)\ge R$. Denote by $S:=\Stab_{\Aut(G)}(x)\cap \Stab_{\Aut(G)}(y)$ and $K:=\Stab_G(x)\cap \Stab_G(y)$ the common stabilizers of $x$ and $y$ in $\Aut(G)$ and $G$, respectively. In particular, $|K|\le N$. To estimate the cardinality of $S$, we observe that $K=S\cap G$. In particular, $K$ is the kernel of the homomorphism from $S$ to $F_\sigma$ induced by the natural projection $\Aut(G)=G\rtimes F_\sigma\to F_\sigma$. Moreover, the image of $S\to F_\sigma$ is a subgroup of the finite group $F_\sigma$. Therefore, $|S|=|K|\cdot|\im(S\to F_\sigma)|\le N|F_\sigma|$, i.e.\ the action of $\Aut(G)$ on $X$ is weakly acylindrical.
\end{proof}

\begin{lem}\label{lem:virtually-cyclic-common}
The groups $G$ and $\Aut(G)$ are not virtually cyclic.
\end{lem}

\begin{proof}
The group $G$ contains standard vertex subgroups isomorphic to $BS(m_i,n_i)$ for all $i\in\mathbb Z/k\mathbb Z$. Since any subgroup of a virtually cyclic group is virtually cyclic, it is enough to see that $BS(m_i,n_i)$ is not virtually cyclic. Recall that $BS(m_i,n_i)=\langle s,t\mid st^{m_i}s^{-1}=t^{n_i}\rangle$ surjects onto $\Z$ by sending $s\mapsto1$ and $t\mapsto0$. By Britton's lemma, $t$ has infinite order, so the infinite subgroup $\langle t\rangle$ has infinite index. This is impossible in a virtually cyclic group, since every infinite subgroup of a virtually cyclic group has finite index. Hence $BS(m_i,n_i)$, and therefore $G$, is not virtually cyclic.

Since $G$ is a subgroup of $\Aut(G)$, the group $\Aut(G)$ cannot be virtually cyclic either.
\end{proof}

\begin{lem}\label{lem:no-global-fixed-point}
The action of $G$ on $X$ has no global fixed point. Consequently, the action of $\Aut(G)$ on $X$ has no global fixed point either.
\end{lem}

\begin{proof}
Assume that $G$ fixes a point $p\in X$, and let $c$ be the unique open cell containing $p$. Since the action is without inversion, every element preserving a cell setwise fixes it pointwise \cite[Section~2.2]{HorHua1}, \cite[Section~1.1]{Marti1}. Hence $G$ fixes the cell $c$ pointwise, so $G\le \Stab_G(c)$.

If $c$ is a $2$-cell, then $\Stab_G(c)=1$, impossible. If $c$ is an edge, then $\Stab_G(c)$ is conjugate to an edge group and hence is infinite cyclic, whereas $G$ contains a vertex subgroup isomorphic to $BS(m_i,n_i)$ and is therefore not cyclic.

Assume now that $c$ is a vertex $v$. Then $G=\Stab_G(v)=G_v$. Since vertex stabilizers are conjugates of the standard vertex groups, this would imply that $G$ is isomorphic to $BS(m_i,n_i)$ for some $i$. On the other hand, every homomorphism $f\colon G\to\mathbb Z$ is trivial. Indeed, applying $f$ to the defining relations of $BS(m_i,n_i)=\langle a_{i},a_{i+1} \mid a_{i}a_{i+1}^{m_{i}}a_{i}^{-1}=a_{i+1}^{n_{i}}\rangle$
gives:
\[
(m_i-n_i)\,f(a_{i+1})=0
\]
for every $i\in\mathbb Z/k\mathbb Z$. Since $m_i\ne n_i$, we obtain $f(a_{i+1})=0$ for every $i$, and hence $f=0$. But $BS(m_i,n_i)$ admits an epimorphism onto $\mathbb Z$, given by $s\mapsto 1$, $t\mapsto 0$. This contradiction shows that $G$ has no global fixed point in $X$.
\end{proof}

\section{Acylindrical hyperbolicity of \texorpdfstring{$\Aut(G)$}{Aut(G)}}\label{sec4}
In this section we prove Theorem~\ref{thm1}, and then deduce Corollaries~\ref{cor:aut-consequences} and~\ref{cor:main-rinfty}.

\subsection{Proof for the generalized Higman groups with \texorpdfstring{$k\ge5$}{k>=5}}

Assume that $G=H_\sigma$ is a generalized Higman group with $k\ge5$.

\begin{prop}\label{prop:aut-ah-generalized}
The group $\Aut(G)$ is acylindrically hyperbolic.
\end{prop}

\begin{proof}
By Section~\ref{sec3}, the complex $X$ is a geodesic Gromov-hyperbolic space, and $\Aut(G)$ acts on $X$ by cellular isometries. By Lemma~\ref{lem:Aut-weak-acyl}, this action is weakly acylindrical. Passing to the second barycentric subdivision $X_\Delta$ of $X$, we obtain a $2$-dimensional simplicial complex which is still $\CAT(-1)$ and piecewise hyperbolic, and has finitely many isometry types of simplices; see \cite[proof of Lemma~5.5]{HorHua1}. The subdivision is taken by geodesic simplices and does not change the underlying metric space. Hence the weak acylindricity constants from Lemma~\ref{lem:Aut-weak-acyl} still apply to the same isometric action, viewed on $X_\Delta$. Therefore \cite[Theorem~E]{MarPrz1} applies and shows that the action of $\Aut(G)$ on the metric space underlying $X_\Delta$, equivalently on $X$, is acylindrical.

We next show that the action of $\Aut(G)$ on $X$ is non-elementary. If it had bounded orbits, then by~\cite[Corollary~II.2.8(1)]{BriHae1} it would have a global fixed point, contrary to Lemma~\ref{lem:no-global-fixed-point}. On the other hand, $\Aut(G)$ is not virtually cyclic by Lemma~\ref{lem:virtually-cyclic-common}. Therefore, by the trichotomy for acylindrical actions on hyperbolic spaces \cite[Theorem~1.1]{Osin1}, the action of $\Aut(G)$ on $X$ is non-elementary.

Hence $\Aut(G)$ admits a non-elementary acylindrical action on a geodesic Gromov-hyperbolic space, so $\Aut(G)$ is acylindrically hyperbolic.
\end{proof}

This proves part~(1) of Theorem~\ref{thm1}.

\subsection{Proof for the \texorpdfstring{$4$}{4}-generated Higman-like groups}

Assume now that
\[
G=H(m_i)_{i\in \Z/4\Z}=H_\sigma
\quad\text{for}\quad
\sigma=\bigl((1,m_1),(1,m_2),(1,m_3),(1,m_4)\bigr),
\]
with $m_i\ge2$ for all $i$.

Recall that an action of a group $\Gamma$ on a $\CAT(0)$ cube complex $Y$ is called \emph{essential} if no $\Gamma$-orbit remains in a bounded neighborhood of a half-space of $Y$, i.e.\ if there do not exist a half-space $\mathfrak h$, a point $x\in Y$ and a constant $R<\infty$ such that the orbit $\Gamma x$ is contained in $\{y\in Y : d(y,\mathfrak h)\le R\}$, see~\cite[p.~852 and Section~3.3]{CapSag1}.

In what follows, $\partial_\infty Y$ denotes the boundary at infinity of a $\CAT(0)$ cubical complex $Y$, i.e.\ the set of asymptotic classes of geodesic rays in $Y$, see~\cite[Chapter~II.8]{BriHae1}.

\begin{martincorg}[\cite{Marti2}]
Let $\Gamma$ be a group acting non-uniformly weakly acylindrically on a $\CAT(0)$ square complex $Y$ such that
\begin{enumerate} 
\item the action is essential,
\item $\Gamma$ has no fixed point in $Y\cup \partial_\infty Y$, and
\item $Y$ is not the product of two unbounded trees.
\end{enumerate} 
Then $\Gamma$ contains strongly contracting elements, every such element satisfies the WPD condition, and in particular $\Gamma$ is either virtually cyclic or acylindrically hyperbolic.
\end{martincorg}

To apply this criterion to the action of $\Aut(G)$ on the CAT(0) square complex $X$, we need to verify the three hypotheses above. A key step is to verify first the corresponding properties for the action of $G$ on $X$. In~\cite[Remark~4.9]{Marti1}, Martin mentions that the action of $G$ on $X$ is essential and that it has no fixed point in $\partial_\infty X$. For completeness, we provide detailed proofs of these statements.

\begin{lem}\label{lem:g-essential}
Let $G$ be a $4$-generated Higman-like group and let $X$ be its associated $\CAT(0)$ square complex. Then the action of $G$ on $X$ is essential.
\end{lem}

\begin{proof}
Fix a hyperplane $H$ of $X$ and let $e$ be an edge dual to $H$. By Martin's construction~\cite[Section~4.2]{Marti1}, there exists an isometrically embedded flat $F\simeq \R^2$, which is a square subcomplex of $X$ isometric to the standard square tiling of the Euclidean plane. Let $C$ be a square containing $e$. Since all squares of $X$ lie in a single $G$-orbit, there exists a translate of $F$ containing $C$; replacing $F$ by that translate, we may assume that $F$ contains $e$. Since $F$ is an isometrically embedded subcomplex, by~\cite[Section~2.6]{Hage1}, $L:=H\cap F$ is nonempty, connected, and separates $F$ into two components. Since each square of $F$ is a square of $X$, the intersection $L$ meets each square of $F$ either trivially or in a single midsegment. Hence $L$ is a hyperplane of the standard square tiling $F$, consisting of midsegments, and therefore a bi-infinite line.

Since $F$ is a Euclidean square complex and $L$ separates $F$ into two half-planes, each component of $F\setminus L$ contains squares $C_k$ with $d_F(C_k,L)\to\infty$. On the other hand, $G$ acts on $X$ with a square as a strict fundamental domain, so for any point $x\in X$, the orbit $Gx$ meets every square of $X$, and, in particular, meets each $C_k$. Therefore every $G$-orbit contains points arbitrarily far from $L$.

To conclude essentiality, it remains to compare distance to $L$ inside $F$ with distance to $H$ inside $X$. Rather than doing this in the $\CAT(0)$ metric, we switch to the combinatorial metrics on $F^{(1)}$ and $X^{(1)}$.

Let $N(H)$ be the carrier of $H$, i.e.\ the union of all closed cubes of $X$ that meet $H$, and let $S:=F\cap N(H)$. Then $S$ is the carrier in $F$ of the line $L$. We will prove that
\begin{equation}\label{eq*}
d_{X^{(1)}}(x,N(H))=d_{F^{(1)}}(x,S)\tag{*}
\end{equation}
for every vertex $x\in F$, using the median graph structure on $X^{(1)}$, whose definition we now recall.

A connected graph $\Gamma$ is called a \emph{median graph} if for every triple of vertices $x,y,z\in \Gamma$ there exists a unique vertex $m=m(x,y,z)$, called the \emph{median} of $x,y,z$, such that $m \in I(x,y)\cap I(y,z)\cap I(z,x)$, where $I(u,v)$ denotes the set of vertices lying on edge-path geodesics from $u$ to $v$. Chepoi proved in~\cite[Theorem~6.1]{Chepo1} that the $1$-skeleton of every $\CAT(0)$ cube complex is a median graph. In particular, $X^{(1)}$ is a median graph.

If $Y$ is a subcomplex of a $\CAT(0)$ cube complex $X$, we say that $Y$ is \emph{combinatorially convex} if $Y^{(1)}$ is convex in $X^{(1)}$ with respect to the edge-path metric, that is, every edge-path geodesic in $X^{(1)}$ between two vertices of $Y^{(1)}$ is contained in $Y^{(1)}$.

Let $x\in F^{(1)}$. The inequality $d_{X^{(1)}}(x,N(H))\le d_{F^{(1)}}(x,S)$ is immediate since $S\subseteq N(H)$. For the reverse inequality, choose a vertex $y\in N(H)$ such that $d_{X^{(1)}}(x,y)=d_{X^{(1)}}(x,N(H))$, and choose any vertex $z\in S=F\cap N(H)$.

Lemma 4.8 of~\cite{Marti1} implies that $F$ is combinatorially convex in $X$, and carriers of hyperplanes are combinatorially convex by \cite[Lemma~1.4]{Marti2}. Hence every geodesic in $X^{(1)}$ from $x$ to $z$ lies in $F$, and every geodesic in $X^{(1)}$ from $y$ to $z$ lies in $N(H)$. Therefore the median $m=m(x,y,z)$ satisfies $m\in I(x,z)\cap I(y,z)\subseteq F\cap N(H)=S$. On the other hand, $m\in I(x,y)$. If $m\neq y$, then $d_{X^{(1)}}(x,m)<d_{X^{(1)}}(x,y)$, while $m\in N(H)$, contradicting the choice of $y$. Thus $m=y$, so $y\in S$. Hence $d_{X^{(1)}}(x,N(H))\ge d_{X^{(1)}}(x,S)$. Since $S\subseteq F$ and $F$ is combinatorially convex, every geodesic in $X^{(1)}$ from $x$ to a vertex of $S$ lies in $F$, so $d_{X^{(1)}}(x,S)=d_{F^{(1)}}(x,S)$. This proves equality~\eqref{eq*}.

We have proved that each of the two half-spaces bounded by $H$ contains vertices arbitrarily far from $N(H)$ in the combinatorial metric of $X^{(1)}$. Finally, since $X$ is finite-dimensional, the combinatorial metric on $X^{(1)}$ and the $\CAT(0)$ metric on $X$ are quasi-isometric; see~\cite[Lemma~2.2]{CapSag1}. Therefore, each of the two half-spaces bounded by $H$ contains, for every $x\in X$, points of the orbit $Gx$ arbitrarily far from $H$ in the $\CAT(0)$ metric. Since $H$ was arbitrary, this means that the action of $G$ on $X$ is essential.
\end{proof}

\begin{lem}\label{lem:g-no-fixpt-on-bdry}
Let $G$ be a $4$-generated Higman-like group and let $X$ be its associated $\CAT(0)$ square complex. Then the action of $G$ on $X$ has no fixed point in $\partial_\infty X$.
\end{lem}

\begin{proof}
Assume for contradiction that $G$ fixes some point $\xi\in \partial_\infty X$. Let $v$ be a vertex of $X$, and let $\rho\colon [0,\infty)\to X$ be the unique $\CAT(0)$ geodesic ray issuing from $v$ and representing $\xi$, which exists by~\cite[Proposition~II.8.2]{BriHae1}. Fix a nontrivial element $g$ of the vertex group $G_v$. We have $gv=v$ and $g\xi=\xi$, so $g\rho$ is again a geodesic ray issuing from $v$ and representing $\xi$. By uniqueness, we get $g\rho=\rho$. Since $g$ preserves distance from $v$, it follows that $g$ fixes $\rho$ pointwise. Thus the fixed-point set $\Fix(g)$ contains the unbounded set $\rho([0,\infty))$. On the other hand, since $g$ is elliptic, \cite[Corollary~2.3]{Marti1} implies that $\Fix(g)$ is either a single vertex or a union of edges contained in the star of a single vertex. In particular, $\Fix(g)$ is bounded. This contradiction proves the lemma.
\end{proof}

Now we can prove the acylindrical hyperbolicity of $\Aut(G)$.

\begin{prop}\label{prop:aut-ah-higman-like} 
The group $\Aut(G)$ is acylindrically hyperbolic.
\end{prop} 
\begin{proof} We check the hypotheses of Martin's Corollary G above for the action of $\Aut(G)$ on $X$. 

First, by Lemma~\ref{lem:Aut-weak-acyl}, the action of $\Aut(G)$ on $X$ is weakly acylindrical, hence, in particular, non-uniformly weakly acylindrical.

Second, by Lemma~\ref{lem:g-essential}, the action of $G$ on $X$ is essential. Since $G \le \Aut(G)$, every $\Aut(G)$-orbit contains a $G$-orbit. Therefore, if an $\Aut(G)$-orbit were contained in a bounded neighborhood of a half-space, then the corresponding $G$-orbit would be as well, contradicting essentiality. Hence the action of $\Aut(G)$ on $X$ is essential.

Third, the action of $G$ has no fixed point in $X\cup \partial_\infty X$, by Lemmas~\ref{lem:no-global-fixed-point} and~\ref{lem:g-no-fixpt-on-bdry}. If $\Aut(G)$ fixed a point of $X\cup\partial_\infty X$, then so would its subgroup $G$, again a contradiction.

Finally, $X$ is not the product of two unbounded trees. Indeed, if a square complex splits as a product of two trees, then the link of every vertex is the simplicial join of two discrete sets, hence a complete bipartite graph. Since by~\cite[Corollary~4.6]{Marti1}, the link of a vertex of $X$ is not a complete bipartite graph, $X$ is not a product of two trees.

Thus all hypotheses of Martin's Corollary G are satisfied. It follows that $\Aut(G)$ is either virtually cyclic or acylindrically hyperbolic. The first possibility is excluded by Lemma~\ref{lem:virtually-cyclic-common}. Therefore $\Aut(G)$ is acylindrically hyperbolic.
\end{proof}

This proves part~(2), and completes the proof of Theorem~\ref{thm1}.\qed

\begin{proof}[Proof of Corollary~\ref{cor:aut-consequences}]
By Theorem~\ref{thm1}, the group $\Aut(G)$ admits a non-elementary acylindrical action on a geodesic Gromov-hyperbolic space $Y$. Section~\ref{sec3} shows that $Z(G)=1$ and that $\Aut(G)=G\rtimes F_\sigma$ for a finite group $F_\sigma$. Hence $G\simeq\Inn(G)$ is a finite-index subgroup of $\Aut(G)$. The restriction of the action to $\Inn(G)$ is still acylindrical. It is also non-elementary: if $g,h\in\Aut(G)$ are independent loxodromic elements, then suitable positive powers of $g$ and $h$ lie in the finite-index subgroup $\Inn(G)$, and these powers remain independent loxodromic elements. Thus $\Inn(G)$, and hence $G$, is acylindrically hyperbolic. This proves part~(1).

For part~(2), the group $G$ is infinite and centerless, and therefore is not virtually central. Theorem~\ref{thm1} and~\cite[Theorem~5.1]{FFaWad1} now imply that the vector space of $\Aut(G)$-invariant homogeneous quasimorphisms on $G$ is infinite-dimensional.

Finally, $G$ is finitely generated, and $Z(G)=1$ implies $[G:Z(G)]=\infty$. Applying~\cite[Corollary~3.7]{CouFou1} to Theorem~\ref{thm1}, we obtain a characteristic subgroup $N\trianglelefteq G$ such that $G/N$ is infinite, $2$-generated, and simple. This proves part~(3).
\end{proof}

\begin{proof}[Proof of Corollary~\ref{cor:main-rinfty}]
For every group $G$ in Theorem~\ref{thm1}, the subgroup $\Inn(G)\simeq G$ is infinite. The conclusion therefore follows from Theorem~\ref{thm1} and Corollary~\ref{cor:aut-ah-implies-rinfty}.
\end{proof}

\section{Comparison of criteria for \texorpdfstring{property $R_\infty$}{property R-infinity}}\label{sec:remarks}

In the course of this work, our original aim was only to construct an isometric action of $\Aut(G)$ on a Gromov-hyperbolic space whose restriction to $\Inn(G)$ is non-elementary, since this is the input needed for Delzant's lemma and hence for property $R_\infty$. However, the geometric arguments available to us led to the stronger conclusion that $\Aut(G)$ is acylindrically hyperbolic. This stronger property is thus not the minimal hypothesis required for our application, but rather the natural output of the proof. It is therefore instructive to compare the following related properties of a group $G$:
\begin{enumerate}
\item $G$ has property $R_\infty$;
\item there exists a nonzero $\Aut(G)$-invariant homogeneous quasimorphism on $G$;
\item there exists an isometric action of $\Aut(G)$ on a geodesic Gromov-hyperbolic space whose restriction to $\Inn(G)$ is non-elementary;
\item $\Aut(G)$ is acylindrically hyperbolic and $\Inn(G)$ is infinite.
\end{enumerate}
We have the following implications:
\[
(4)\Longrightarrow (2)\Longrightarrow (1)
\qquad\text{and}\qquad
(4)\Longrightarrow (3)\Longrightarrow (1).
\]
Here $(4)\Longrightarrow(2)$ follows from \cite[Theorem~E]{FFaWad1}, $(2)\Longrightarrow(1)$ is~\cite[Corollary~8.1.4]{FFIMSW1}, $(4)\Longrightarrow(3)$ follows from~\cite[Lemma~7.2]{Osin1}, and $(3)\Longrightarrow(1)$ is our Corollary~\ref{cor:action-rinfty} based on Delzant's lemma.

It turns out that the implications displayed above are the only implications among properties~(1)--(4) that hold in general. In an earlier version of the paper, we asked whether property~(3) implies property~(2). Francesco Fournier-Facio informed us that for any prime $p$, the group $\PSL_2(\Z[1/p])$ gives a counterexample, i.e.\ it satisfies property~(3) but not property~(2). The proof uses properties of arithmetic lattices and super-rigidity and is beyond the scope of this article.

We now record the failures of all other converse implications.

$(1)\centernot\Longrightarrow(2)$: Let $G=N_{2,4}$ be the free nilpotent group of rank $2$ and class $4$. By~\cite{DekGon1}, $G$ has property $R_\infty$. Since $G$ is amenable, every homogeneous quasimorphism on $G$ is a homomorphism by~\cite[Proposition~2.65]{Caleg1}. But any $\Aut(G)$-invariant homomorphism $h\colon G\to\R$ must vanish. Indeed, $h$ factors through the abelianization $G^{ab}\simeq\Z^2$, and the automorphism sending each free generator to its inverse induces $-\id$ on $\Z^2$. Thus the induced homomorphism $\bar h\colon\Z^2\to\R$ satisfies $\bar h=\bar h\circ(-\id)$, hence $\bar h=-\bar h$, so $\bar h=0$. Thus $G$ has property $R_\infty$ but admits no nonzero $\Aut(G)$-invariant homogeneous quasimorphisms.

$(1)\centernot\Longrightarrow(3)$: Consider $G=N_{2,4}$, as above. Then $\Inn(G)\simeq G/Z(G)$ is nilpotent as well. If $\Inn(G)$ admitted a non-elementary action on a geodesic Gromov-hyperbolic space, then it would contain two loxodromic elements generating a free non-abelian subgroup; see, for example,~\cite[Proposition~28]{SorVas1}. Since a nilpotent group cannot contain a free non-abelian subgroup, no such action exists. Thus property~(3) fails for $G$.

$(1)\centernot\Longrightarrow(4)$ follows from $(1)\centernot\Longrightarrow(2)$ and $(4)\Longrightarrow(2)$.

$(2)\centernot\Longrightarrow(3)$: Let $G=BS(1,m)=\langle a,b\mid b^{-1}ab=a^m\rangle$, with $|m|>1$. In~\cite[Section~2]{Levit1} Levitt defines for arbitrary generalized Baumslag--Solitar groups the modular homomorphism $\Delta\colon G\to\Q^\times$ and observes that $\Delta\circ\alpha=\Delta$ for every $\alpha\in\Aut(G)$; moreover, for $G=BS(1,m)$ the image of $\Delta$ is generated by $1/m$. In particular, $\chi\colon G\to\Z$ defined by $\chi(g):=-\log_{|m|}|\Delta(g)|$ is a nonzero homomorphism, hence a nonzero homogeneous quasimorphism, and it is $\Aut(G)$-invariant, so property~(2) holds for $G$. If property~(3) held for $G$, then $\Inn(G)$ would admit an action on a geodesic Gromov-hyperbolic space with two independent loxodromic elements. Again, by~\cite[Proposition~28]{SorVas1}, sufficiently large powers of these two loxodromics would generate a free non-abelian subgroup of $\Inn(G)$. Since $G$ and $\Inn(G)\simeq G/Z(G)$ are solvable, this is impossible. Thus property~(3) fails for $G$.

$(2)\centernot\Longrightarrow(4)$ follows from $(2)\centernot\Longrightarrow(3)$ and $(4)\Longrightarrow(3)$. Another example is given by $G=F_2\times F_2$: this follows from~\cite[Theorem~5.5]{FFaWad1} and~\cite[Theorem~1.5]{Genev1}, as pointed out in~\cite[Section~6.1]{FFaWad1}.

$(3)\centernot\Longrightarrow(4)$: Let $G=H_4\times \Z^2$, where $H_4$ is the $4$-generator Higman group. Since $Z(H_4)=1$, we have $Z(G)=\Z^2$ and therefore $\Inn(G)\simeq G/Z(G)\simeq H_4$. As the center is a characteristic subgroup, every automorphism of $G$ preserves $Z(G)$, so there is a natural homomorphism $\rho\colon\Aut(G)\to \Aut(G/Z(G))\times \Aut(Z(G))\simeq\Aut(H_4)\times \GL_2(\Z)$. Its kernel consists exactly of those automorphisms $\varphi$ which induce identity on $G/Z(G)$ and on $Z(G)$, i.e.\ automorphisms of the form $\varphi(h,z)=(h,z+\lambda(h))$, for some map $\lambda\colon H_4\to\Z^2$. Moreover,  the condition that $\varphi$ is a homomorphism is equivalent to $\lambda$ being a homomorphism itself. Since $H_4^{ab}$ is trivial, $\Hom(H_4,\Z^2)=0$, hence $\lambda=0$, which means that $\ker\rho=\id$. Every pair $(\alpha,A)\in \Aut(H_4)\times \GL_2(\Z)$ gives an automorphism of $G$, sending $(h,z)\mapsto(\alpha(h),Az)$, so $\rho$ is onto, and hence $\Aut(G)\simeq \Aut(H_4)\times \GL_2(\Z)$. Let $Y$ be a geodesic Gromov-hyperbolic space on which $\Aut(H_4)$ acts non-elementarily and acylindrically, whose existence follows from Proposition~\ref{prop:aut-ah-higman-like}. Let $\Aut(G)$ act on $Y$ via the projection $\Aut(G)\to\Aut(H_4)$. Then the restriction of this action to $\Inn(G)$ is non-elementary by Proposition~\ref{prop:aut-ah-higman-like} and \cite[Lemma~7.2]{Osin1}, so property (3) holds for $G$. On the other hand, $\Aut(G)$ is a direct product of two infinite groups, hence it is not acylindrically hyperbolic by~\cite[Corollary~7.3\,(b)]{Osin1}. Thus $(3)\centernot\Longrightarrow(4)$.

Finally, we justify the hypothesis that $\Inn(G)$ should be infinite, in property~(4). Let $G=\Z^2$. Then $\Inn(G)=1$, while $\Aut(G)\simeq \GL_2(\Z)$. Since $\GL_2(\Z)$ is virtually free and not virtually cyclic, it is a non-elementary hyperbolic group, hence acylindrically hyperbolic by~\cite[Theorem~1.2]{Osin1}. On the other hand, $G$ does not have property $R_\infty$: for the automorphism $\varphi=-\id_{\Z^2}$, the group $G$ being abelian implies that the $\varphi$-twisted conjugacy classes are the cosets of $(1-\varphi)G=2\Z^2$, and hence $R(\varphi)=\bigl|\Z^2:2\Z^2\bigr|=4$. Therefore the assumption that $\Inn(G)$ be infinite is essential in~(4).


\frenchspacing
\begin{thebibliography}{FFIMSW26}

\bibitem[Bal20]{Balas1}
{\bf S H Balasubramanya,}
{\it Finite extensions of $\mathcal H$- and $\mathcal{AH}$-accessible groups,}
Topology Proc. 56 (2020), 297--304.

\bibitem[BF02]{BesFuj1} 
{\bf M Bestvina, K Fujiwara,} 
{\it Bounded cohomology of subgroups of mapping class groups,} 
Geom. Topol. 6 (2002), 69--89.

\bibitem[BH99]{BriHae1} 
{\bf M\,R Bridson, A Haefliger,} 
{\it Metric spaces of non-positive curvature,} 
Grundlehren der mathematischen Wissenschaften, 319. Springer-Verlag, Berlin, 1999. 

\bibitem[Cal09]{Caleg1}
{\bf D Calegari,}
{\it scl,}
MSJ Memoirs, vol.~20, Mathematical Society of Japan, Tokyo, 2009.

\bibitem[CS22]{CalSor1} 
{\bf M Calvez, I Soroko,} 
{\it Property $R_\infty$ for some spherical and affine Artin--Tits groups,} 
J. Group Theory 25 (2022), no. 6, 1045--1054.

\bibitem[CS11]{CapSag1} 
{\bf P-E Caprace, M Sageev,} 
{\it Rank rigidity for $\CAT(0)$ cube complexes,} 
Geom. Funct. Anal. 21 (2011), no. 4, 851--891.

\bibitem[Che00]{Chepo1} 
{\bf V Chepoi,} 
{\it Graphs of some $\CAT(0)$ complexes,} 
Adv. Appl. Math. 24 (2000), 125--179.

\bibitem[CFF26]{CouFou1} 
{\bf R Coulon, F Fournier-Facio,} 
{\it Infinite simple characteristic quotients,} 
J. Éc. polytech. Math. 13 (2026), 1463--1481.

\bibitem[DG14]{DekGon1} 
{\bf K Dekimpe, D Gon\c{c}alves,} 
{\it The $R_\infty$ property for free groups, free nilpotent groups and free solvable groups,} 
Bull. London Math. Soc. 46 (2014), no. 4, 737--746.

\bibitem[DL24]{DekLat1} 
{\bf K Dekimpe, M Lathouwers,} 
{\it $R_\infty$-property for finitely generated torsion-free $2$-step nilpotent groups of small Hirsch length,} 
Topology Appl. 359 (2025) Paper No. 109084, 12 pp.

\bibitem[DLR26]{DeLiRe1} 
{\bf K Dekimpe, P\,M\,Lins\,de\,Araujo, Y\,S Rego,} 
{\it Cohomological and quasi-isometric diversity of groups with property $R_\infty$,} 
preprint, 2026, arXiv:2602.17411.

\bibitem[Fel01]{Felsh1}
{\bf A L Fel'shtyn,}
{\it The Reidemeister number of any automorphism of a Gromov hyperbolic group is infinite.} 
Zap. Nauchn. Sem. S.-Peterburg. Otdel. Mat. Inst. Steklov. (POMI) 279 (2001), Geom. i Topol. 6, 229--240, 250; reprinted in J. Math. Sci. (N.Y.) 119 (2004), no. 1, 117--123 .

\bibitem[FGD10]{FeGoDa1} 
{\bf A Fel'shtyn, D\,L Gon\c{c}alves,} 
{\it Twisted conjugacy classes in symplectic groups, mapping class groups and braid groups,} 
with an appendix written jointly with Francois Dahmani. Geom. Dedicata 146 (2010), 211--223.

\bibitem[FN16]{FelNas1} 
{\bf A Fel'shtyn, T Nasybullov,} 
{\it The $R_{\infty}$ and $S_{\infty}$ properties for linear algebraic groups,} 
J. Group Theory 19 (2016), no. 5, 901--921.

\bibitem[FT15]{FelTro2} 
{\bf A Fel'shtyn, E Troitsky,} 
{\it Aspects of the property $R_{\infty}$,} 
J. Group Theory 18 (2015), no. 6, 1021--1034.

\bibitem[FFIMSW26]{FFIMSW1} 
{\bf F Fournier-Facio, H Iveson, A Martino, W Sgobbi, P Wong,} 
{\it Property $R_\infty$ for groups with infinitely many ends,} 
Geom. Dedicata 220 (2026), article number 25.

\bibitem[FFW23]{FFaWad1} 
{\bf F Fournier-Facio, R\,D Wade,} 
{\it Aut-invariant quasimorphisms on groups,} 
Trans. Amer. Math. Soc., 376 (10) (2023), 7307--7327.

\bibitem[Gen24]{Genev1} 
{\bf A Genevois,} 
{\it Automorphisms of graph products of groups and acylindrical hyperbolicity}, 
Mem. Amer. Math. Soc. 301 (2024), no.~1509, 127 pp.

\bibitem[EH24]{EscHor1} 
{\bf A Escalier, C Horbez,} 
{\it Graph products and measure equivalence: classification, rigidity, and quantitative aspects,} 
preprint, 2024, arXiv:2401.04635.

\bibitem[Gen19]{GenOne1} 
{\bf A Genevois,} 
{\it Negative curvature in automorphism groups of one-ended hyperbolic groups,} 
J. Comb. Algebra 3 (2019), no.~3, 305--329.

\bibitem[GH21]{GenHor1} 
{\bf A Genevois, C Horbez,} 
{\it Acylindrical hyperbolicity of automorphism groups of infinitely ended groups,} 
J. Topol. 14 (2021), no.~3, 963--991.

\bibitem[GM19]{GenMar1} 
{\bf A Genevois, A Martin,} 
{\it Automorphisms of graph products of groups from a geometric perspective,} 
Proc. Lond. Math. Soc. (3) 119 (2019), no.~6, 1745--1779.

\bibitem[Gro87]{Gromo1} 
{\bf M Gromov,} 
{\it Hyperbolic groups,} 
Essays in group theory, 75--263, Math. Sci. Res. Inst. Publ., 8, Springer, New York, 1987.

\bibitem[Hag14]{Hage1} 
{\bf M F Hagen,} 
{\it Weak hyperbolicity of cube complexes and quasi-arboreal groups,} 
J. Topol., 7 (2014), no.~2, 385--418.

\bibitem[Hig51]{Higma1} 
{\bf G Higman,} 
{\it A finitely generated infinite simple group,} 
J. London Math. Soc. 26 (1951), 61--64.

\bibitem[HH25]{HorHua1} 
{\bf C Horbez, J Huang,} 
{\it Measure equivalence rigidity among the Higman groups,} 
J. Eur. Math. Soc. (2025), published online first. \href{https://doi.org/10.4171/jems/1614}{doi:10.4171/jems/1614}

\bibitem[JMS25]{JonManSar1} 
{\bf O Jones, G Mangioni, G Sartori,} 
{\it JSJ splittings for all Artin groups,} 
J. Algebra 701 (2026), 436--476.

\bibitem[KM14]{Kourovka1} 
{\bf E\,I Khukhro, V\,D Mazurov,} 
{\it Unsolved problems in group theory. The Kourovka Notebook,} 
preprint, 2014, arXiv:1401.0300.

\bibitem[Lev07]{Levit1} 
{\bf G Levitt,} 
{\it On the automorphism group of generalized Baumslag--Solitar groups,} 
Geom. Topol. 11 (2007), 473--515.

\bibitem[LL00]{LevLus1} 
{\bf G Levitt, M Lustig,} 
{\it Most automorphisms of a hyperbolic group have very simple dynamics,} 
Ann. Sci. \'Ecole Norm. Sup. (4) 33 (2000), no. 4, 507--517.

\bibitem[Mar17]{Marti1} 
{\bf A Martin,} 
{\it On the cubical geometry of the Higman group,} 
Duke Math. J. 166 (2017), no. 4, 707--738.

\bibitem[Mar21]{Marti2} 
{\bf A Martin,} 
{\it Acylindrical actions on $\CAT(0)$ square complexes,} 
Groups Geom. Dyn. 15 (2021), no. 1, 335--369.

\bibitem[MP22]{MarPrz1} 
{\bf A Martin, P Przytycki,} 
{\it Acylindrical actions for two-dimensional Artin groups of hyperbolic type,} 
Int. Math. Res. Not. IMRN 2022, no. 17, 13099--13127.

\bibitem[MO15]{MinOsi1} 
{\bf A Minasyan, D Osin,} 
{\it Acylindrical hyperbolicity of groups acting on trees,} 
Math. Ann. 362 (2015), no.~3--4, 1055--1105.

\bibitem[MO19]{MinOsiCorr1} 
{\bf A Minasyan, D Osin,} 
{\it Correction to: Acylindrical hyperbolicity of groups acting on trees,} 
Math. Ann. 373 (2019), no.~1--2, 895--900.

\bibitem[Mon17]{Monod1} 
{\bf N Monod,} 
{\it Variations on a theme by Higman,} 
Expo. Math. 35 (2017), no.~2, 226--235.

\bibitem[Osi16]{Osin1} 
{\bf D Osin,} 
{\it Acylindrically hyperbolic groups,} 
Trans. Amer. Math. Soc. 368 (2016), no.~2, 851--888.

\bibitem[Sor26]{Sorok1}
{\bf I Soroko,}
{\it Property $R_\infty$ for Artin and Coxeter groups whose outer automorphisms are length-preserving,}
preprint, 2026, arXiv:2610.13904.

\bibitem[SV24]{SorVas1} 
{\bf I Soroko, N Vaskou,} 
{\it Property $R_\infty$ for new classes of Artin groups,} 
preprint, 2024, arXiv:2409.18123v2.

\bibitem[Tro25]{Troit1} 
{\bf E Troitsky,} 
{\it Twisted conjugacy in residually finite groups of finite Prüfer rank,} 
J. Group Theory 28 (2025), no. 1, pp. 151--164.

\bibitem[Wit23]{Witdo1} 
{\bf T Witdouck,} 
{\it The $R_\infty$-property for right-angled Artin groups and their nilpotent quotients,} 
preprint, 2023, arXiv:2304.01077.

\end{thebibliography}
\end{document}